\documentclass[a4paper,11pt]{scrartcl}
\usepackage[utf8]{inputenc}
\RequirePackage[T1]{fontenc}
\usepackage{graphicx} 
\usepackage{amssymb,amsmath,enumitem}
\usepackage{amsthm}
\usepackage{thmtools,thm-restate}
\usepackage{microtype}
\usepackage{titling}

\usepackage{authblk}
\usepackage{xcolor}
\usepackage{tikz}
\usepackage[hyphens]{xurl}
\usepackage[colorlinks=true, linkcolor = violet, citecolor = blue, urlcolor = blue]{hyperref}
\hypersetup{
    pdftitle = {Bounds on treewidth via excluding disjoint unions of cycles},
	pdfauthor = {Meike Hatzel, Chun Hung Liu, Bruce Reed, Sebastian Wiederrecht}
}

\declaretheorem[name=Lemma, numberwithin = section]{lemma}
\declaretheorem[name=Theorem,sibling = lemma]{theorem}

\usepackage[noabbrev,capitalise,nameinlink]{cleveref}
\crefname{claim}{Claim}{Claims}
\crefname{lemma}{Lemma}{Lemmas}
\crefname{theorem}{Theorem}{Theorems}
\crefname{proposition}{Proposition}{Propositions}
\crefname{question}{Question}{Questions}
\crefname{conjecture}{Conjecture}{Conjectures}
\crefname{figure}{Figure}{Figures}
\crefname{corollary}{Corollary}{Corollaries} 
\crefformat{equation}{(#2#1#3)}
\Crefformat{equation}{Equation #2(#1)#3}

\newcommand\abs[1]{\lvert#1\rvert}
                       
\renewcommand{\epsilon}{\varepsilon}  
      
\newcommand{\Beta}{\mathcal{B}}

\DeclareMathOperator{\ord}{ord}
\DeclareMathOperator{\tw}{tw}

\title{Bounds on treewidth via excluding disjoint unions of cycles} 

\predate{}
\date{}
\postdate{}

\preauthor{}
\DeclareRobustCommand{\authorthing}{
    \begin{center}
	    Meike Hatzel\thanks{Department of Mathematics, Technical University Darmstadt, Germany, \href{mailto:research@meikehatzel.com}{research@meikehatzel.com}.
        Meike Hatzel's research was partly supported by the Federal Ministry of Education and Research (BMBF), by a fellowship within the IFI programme of the German Academic Exchange Service (DAAD) and by the Institute for Basic Science (IBS-R029-C1).} \hspace{1cm}
        Chun-Hung Liu\thanks{Department of Mathematics, Texas A\&M University, USA. \href{mailto:chliu@tamu.edu}{chliu@tamu.edu}. Partially supported by NSF under CAREER award DMS-2144042.} \hspace{1cm}
        Bruce Reed\thanks{Mathematical Institute, Academica Sinica, Taiwan, \href{mailto:bruce.al.reed@gmail.com}{bruce.al.reed@gmail.com}. Supported by  NSTC Grant 112-2115-M-001-013 -MY3.} \hspace{1cm}
        Sebastian Wiederrecht\thanks{School of Computing, KAIST, South Korea, \href{mailto:wiederrecht@kaist.ac.kr}{wiederrecht@kaist.ac.kr}.}
\end{center}}
\author{\authorthing}
\postauthor{}

\begin{document}
\maketitle

\begin{abstract} 
    One of the fundamental results in graph minor theory is that for every planar graph~$H$,
    there is a minimum integer~$f(H)$ such that graphs with no minor isomorphic to~$H$ have treewidth at most~$f(H)$. 
    The best bound known for an arbitrary planar $H$ is ${O(\abs{V(H)}^9\operatorname{poly~log}\abs{V(H)})}$.
    We show that if $H$ is the disjoint union of cycles, then $f(H)$ is $O(|V(H)|\log^2 |V(H)|)$, which is a $\log|V(H)|$ factor away from being optimal. 
\end{abstract}

 \section{Introduction}

Treewidth measures the similarity of a given graph to a tree.
It is of both algorithmic and structural importance and was independently discovered multiple times~\cite{treewidth72,treewidth76,treewidth84}.
Many algorithmically hard problems are fixed-parameter tractable with treewidth as the parameter; in particular, Courcelle's theorem~\cite{courcelle} states that every property expressible in monadic second-order logic can be determined in linear time for graphs of bounded treewidth.
More precisely, the running time of those algorithms is of the form $g(w)n^{O(1)}$, where $g$ is a function and $w$ is the treewidth of the input graph.
Hence, obtaining quantitatively good bounds for the maximum treewidth of graphs in specific graph classes is of interest.

Robertson and Seymour~\cite{RS1991} proved that for every planar graph~$H$,
there is an integer $n_H$ such that graphs that do not contain~$H$ as a minor have treewidth at most $n_H$. 
We define $f(H)$ as the smallest $n_H$ for which this is true. 
The best upper bound known on $f(H)$ for arbitrary $H$, obtained by Chuzhoy and Tan~\cite{CT2020}, is~$O(\abs{V(H)}^9 \operatorname{poly~log} \abs{V(H)})$ as a combined consequence of their result and a result of Robertson, Seymour, and Thomas~\cite{RST1994} that every planar~$H$ is a minor of a ${(k \times k)}$-grid for~${k = 2\abs{V(H)}}$. 

It is natural to ask for a better bound on~$f(H)$ for~$H$ in various classes of planar graphs.
A trivial lower bound for $f(H)$ is~${\Omega(\abs{V(H)})}$ since the complete graph on~${\abs{V(H)}-1}$ vertices has treewidth~${\abs{V(H)}-2}$ and does not contain~$H$ as a minor.
Another lower bound comes from counting the number of disjoint cycles.
It is well-known (for example, see~\cite{graphminorsV}) that there are $n$-vertex graphs with girth $\Omega(\log n)$ and treewidth $\Omega(n)$ for infinitely many integers $n$.
Such graphs have $O(n/\log n)$ disjoint cycles and hence cannot contain any graph $H$ with more than $O(n/\log n)$ disjoint cycles as a minor.
This implies that if $H$ contains a spanning subgraph that is a disjoint union of cycles of length $O(1)$, then $f(H)$ is $\Omega(|V(H)| \log |V(H)|)$.
In fact, if $H$ is a disjoint union of cycles of length at most a fixed constant, then $f(H)$ is $\Theta(|V(H)| \log |V(H)|)$, where the upper bound follows from the Erd\H{o}s-P\'{o}sa theorem for cycles of length at least a fixed constant~\cite{EPlongcycles2014,longCycles2017} since every graph that can be made a graph with no cycle of length at least $\ell$, which has treewidth at most $\ell-2$~\cite{birmele2003,FL1989}, by deleting at most $w$ vertices has treewidth at most $w+\ell-2$.

This paper focuses on $f(H)$ for graphs~$H$ that are disjoint unions of cycles without the $O(1)$ upper bound on the length of the cycles.

Fellows and Langston~\cite{FL1989} and Birmele~\cite{birmele2003} showed that if~$H$ is a cycle, then~$f(H)={\abs{V(H)}-2}$.
Mousset, Noever, Škorić, and Weissenberger~\cite{longCycles2017} proved the tight bound of $\Theta(r\ell + r \log{r})$ in the case that all $r$ cycles are of length $\ell$.
Gollin, Hendrey, Oum, and Reed~\cite{GHOR2024} proved that if $H$ is the disjoint union of two cycles, then $f(H)=(1+o(1))|V(H)|$ while if $H$ is the disjoint union of $o(\frac{\sqrt{|V(H)}}{\sqrt{\log |V(H)|}})$ cycles, then $f(H)$ is at most $\frac{3|V(H)|}{2} + o(|V(H)|)$.

\begin{theorem}[\cite{GHOR2024}]
    \label{fewcycles}
    There is an absolute constant~$c$ such that for every~${r \geq 3}$, 
    if~$H$ is the disjoint union of~$r$ cycles, then 
    \[
        {f(H) \leq \frac{3\abs{V(H)}}{2} + cr^2 \log r}.
    \]
    If~$H$ is the disjoint union of two cycles, then 
    \[
        {f(H) < \abs{V(H)} + \frac{9}{2} \lceil \sqrt{4 + \abs{V(H)}} \rceil + 2}.
    \]
\end{theorem}

With the trivial upper bound of ${r\leq \frac{1}{3}|V(H)|}$, \cref{fewcycles} implies that $f(H)$ is $O(|V(H)|^2\log |V(H)|)$ for general unions of cycles $H$.
An improvement to $f(H)=O(|V(H)|^2)$ can be obtained from the aforementioned result of \cite{longCycles2017} by taking $r=\ell=|V(H)|$.
In this paper, we further improve this bound to $O(|V(H)|\log^2 |V(H)|)$; this is a consequence of the following more fine-grained bound that takes into account not only the number $r$ of cycles whose union make up the graph $H$ but also an upper bound $\ell$ on the length of these cycles.

\begin{theorem}
    \label{maintheorem}
    There is an absolute constant $c$ such that if $H$ is the disjoint union of $r$ cycles of length at most $\ell$, then
    \[
    f(H) \le c|V(H)|\log (r+1)  +cr\log r \log \ell.
    \] 
\end{theorem} 

Note that $\ell$ can be dependent on $|V(H)|$.
If $H$ is a disjoint union of cycles of the same length, then $r=\frac{|V(H)|}{\ell}$ and \cref{maintheorem} give \(f(H) \leq c|V(H)|\log (r+1)  +c|V(H)| \log r \cdot \frac{\log\ell}{\ell} = O(|V(H)| \log r)\).
When $\ell=O(1)$, \cref{maintheorem} recovers the bound $O(|V(H)|\log|V(H)|)$ stated above.

\section{Preliminaries}

All graphs in this paper are simple and finite, and all logarithms are considered base 2.
In proving our result, we focus on a concept dual to treewidth: the \emph{bramble number}. 
A \emph{bramble}~$\Beta$ in a graph $G$ is a family of connected subgraphs of $G$, every two of which intersect or are joined by an edge. 
A \emph{hitting set} for a bramble is a set of vertices intersecting all of its elements. 
The \emph{order} of a bramble~$\Beta$, denoted~$\ord(\Beta)$, is the minimum size of a hitting set for~$\Beta$. 
The \emph{bramble number of~$G$} is the largest order of a bramble in~$G$.
Any subset of a bramble~$\Beta$ is a bramble, called a \emph{subbramble} of~$\Beta$. 
Seymour and Thomas~\cite{ST1993} showed the following, see \cite{Reed1997}. 

\begin{theorem}[\cite{ST1993}]\label{duality}
    The treewidth of a graph is exactly one less than its bramble number. 
\end{theorem}

We observe that for every vertex $v$ in $G$, the treewidth of~$G-v$ is at least the treewidth of~$G$ minus 1 because given a bramble of order~$b$ in~$G$, deleting the bramble elements containing $v$ yields a bramble of order at least $b-1$ in $G-v$. 
More generally, for every subset $X$ of $V(G)$
and maximum order bramble $\Beta$ of $G$ if the subbramble of $\Beta$ consisting of its 
elements which intersect $X$ has order $a$, then the treewidth of $G$ is at most $a$ more than the treewidth of $G-X$. 

For a bramble $\Beta$ in a graph $G$ and a subset $X$ of $V(G)$, we define $\Beta_X$ as the set consisting of the elements of $\Beta$ intersecting $X$.
Hence, the above observation implies the following lemma.

\begin{lemma} \label{small_hit}
Let $G$ be a graph.
Let $\Beta$ be a bramble in $G$ of maximum order.
If $X \subseteq V(G)$ such that $\Beta_X$ has order at most $k$, then $\tw(G) \leq \tw(G-X)+k$.
\end{lemma}

We also need results about brambles hit by cycles (\cref{cyclebramble}) and paths (\cref{lem:pathpartition}).  

\begin{lemma}[{\cite[Theorem 2.4]{BBR2007a}}]
    \label{cyclebramble}
    Let~$G$ be a graph having a bramble~$\Beta$ of order at least three. 
    Then, there is a cycle~$C$ meeting every element of~$\Beta$.
\end{lemma}

\begin{lemma}
    \label{menger}
    Let~$G$ be a graph with a bramble~$\Beta$. 
    Let~$S$ and~$T$ be two subsets of $V(G)$ such that $\Beta_S$ and $\Beta_T$ have order at least~$\ell$.
    Then there are~$\ell$ disjoint paths in $G$ from~$S$ to~$T$.
\end{lemma}

\begin{proof}
    Suppose towards a contradiction that the desired $\ell$ disjoint paths between $S$ and $T$ do not exist.  
    By Menger's Theorem, there is a cutset~$X$ of size less than~$\ell$ separating~$S$ from~$T$.
    Now, there exists an element~$B$ of~$\Beta$ disjoint from~$X$ because the order of~$\Beta$ is at least~$\ell$. 
    Since~$G[B]$ is connected, one of $S$ or $T$ does not intersect the component of~${G-X}$ containing~$B$. 
    By symmetry, we may assume that~$S$ does not intersect the component of~${G-X}$ containing~$B$. 
    Since all elements of~$\Beta_S$ either intersect~$B$ or are joined by an edge to~$B$, they all intersect~$X$. 
    But then~$X$ is a hitting set for~$\Beta_S$ and therefore~$\Beta_S$ has order at most~${\abs{X} < \ell}$, which is a contradiction. 
\end{proof}

Combining these two lemmas, we obtain the following result, which is similar to the $\ell=2$ case of \cite[Lemma 3.2]{REED2012}:

\begin{lemma}
    \label{lem:pathpartition}
    Let~$t \geq 2$ be a positive integer and $\Beta$ a bramble of order at least $2t$ in a graph~$G$.
    Then there exist
    \begin{enumerate}
        \item two disjoint paths $P_1$ and $P_2$ in $G$ such that both $\Beta_{V(P_1)}$ and $\Beta_{V(P_2)}$ have order exactly $t$, and 
        \item disjoint paths $Q_1, \dots , Q_t$ each of which has an endpoint on each $P_i$ and is internally disjoint from $P_1 \cup P_2$. 
    \end{enumerate}
\end{lemma}

\begin{proof}
    We apply \cref{cyclebramble} to obtain a cycle $C$ intersecting every element of $\Beta$. 
    We let $P_1$ be a subpath of $C$ such that $\Beta_{V(P_1)}$ has order at least $t$ and is minimal with this property.   
    So, $\Beta_{V(P_1)}$ has order exactly $t$, and hence $C-V(P_1)$ is a path such that $\Beta_{V(C-V(P_1))}$ has order at least $t$. 
    We let $P_2$ be a subpath of $C-V(P_1)$ such that $\Beta_{V(P_2)}$ has order at least $t$ and is minimal with this property. 
    So, $\Beta_{V(P_2)}$ has order exactly $t$.
    \cref{menger} implies there are $t$ disjoint paths from $V(P_1)$ to $V(P_2)$. We choose these to minimise their total length so each is internally 
    disjoint from $P_1 \cup P_2$. 
\end{proof}

We will use the paths in \cref{lem:pathpartition} to construct disjoint cycles (\cref{lem:contracting_linkage_between_paths}), which relies on the following famous theorem of Erd\H{o}s and P\'{o}sa~\cite{EP1965}.

\begin{theorem}[\cite{EP1965}]\label{thm:ErdosPosa}
    There is a constant~$c^* \geq 1$ such that, for every positive integer~$k$, every graph contains either $k$ disjoint cycles or a set of vertices of size at most~${c^* k \log k}$ which hits every cycle. 
\end{theorem}

\begin{lemma} \label{lem:EPsubcubic}
Let $G$ be a graph of maximum degree at most three.
Let $k$ be a positive integer.
If $|E(G)| \geq |V(G)|+3c^*k\log k$, where $c^*$ is the constant in \cref{thm:ErdosPosa}, then $G$ contains at least $k$ disjoint cycles.
\end{lemma}

\begin{proof}
Suppose, to the contrary, that $G$ contains at most $k-1$ disjoint cycles.
Then, a set of at most $c^*k\log k$ vertices in $G$ hitting all cycles in $G$ exists by~\cref{thm:ErdosPosa}.
Since $G$ has maximum degree at most three, there is a set $S$ of at most $3c^*k\log k$ edges hitting all cycles in $G$.
So $G-S$ is a graph with $|E(G-S)| \geq |V(G)| = |V(G-S)|$ with no cycle, a contradiction.
\end{proof}

\begin{lemma}
    \label{lem:contracting_linkage_between_paths}
    Let $k$ be a positive integer and $G$ be a graph.
    Let $P_1$ and $P_2$ be disjoint paths in $G$.
    Let $Q_1,\dots,Q_{\ell}$ be disjoint paths in $G$ between $V(P_1)$ and $V(P_2)$ internally disjoint from $V(P_1 \cup P_2)$.
    If $\ell \geq 2+3c^*k\log k$, where $c^*$ is the constant in \cref{thm:ErdosPosa}, then $P_1 \cup P_2 \cup \bigcup_{i=1}^{\ell} Q_i$ contains at least $k$ disjoint cycles, each containing at least two paths in $\{Q_1,\dots Q_\ell\}$.
\end{lemma}

\begin{proof}
    We obtain an auxiliary graph $J$ from $P_1 \cup P_2 \cup \bigcup_{i=1}^{\ell} Q_i$ by contracting each $Q_i$ into a single edge.
    Then $J$ is a graph of maximum degree at most three and $|E(J)| = |E(P_1)|+|E(P_2)|+\ell = |V(J)|-2+\ell  \geq |V(J)| +3c^* k\log k$.
    By~\cref{lem:EPsubcubic}, $J$ contains at least $k$ disjoint cycles.
    Note that each of those cycles contains at least two edges not in $E(P_1 \cup P_2)$.
    By replacing the contracted edges with the original paths, this yields at least $k$ disjoint cycles in $P_1 \cup P_2 \cup \bigcup_{i=1}^{\ell} Q_i$, each containing at least two paths in $\{Q_1,\dots Q_\ell\}$.
\end{proof}

\section{The proof of \texorpdfstring{\cref{maintheorem}}{Theorem 1.2}}

We now have everything in place to prove our main theorem.

We proceed by induction, with three different induction steps. 
Which step we apply depends on the length $\ell_1$ of the longest cycle $C_1$ of $H$ and its relationship to the number $r$ of cycles in $H$. 

If there is a cycle $C$ in $G$ of length at least $\ell_1$ such that $\Beta_{V(C)}<6\ell_1$ (Case 1), then we delete it and apply induction on $H-C_1$ in $G-C$.
We note that if we could always apply this step, we would get that $f(H)<6|V(H)|$; but it is not true in general because $f(H)$ is $\Omega(|V(H)|\log|V(H)|)$ as we discussed in the introduction.

Otherwise, if $\ell_1$ is large in terms of $r$ (Case 2.1), then a straightforward argument shows that we can find all the cycles with one application 
of \cref{lem:contracting_linkage_between_paths}.
Indeed, this case could also be handled by applying \cref{fewcycles} from \cite{GHOR2024}.

The remaining case (Case 2.2) is more delicate.
It includes the case that there are more than $\frac{\sqrt{n}}{\log n}$ cycles, and a more nuanced application of \cref{lem:contracting_linkage_between_paths} is required. 

\begin{proof}[Proof of~\cref{maintheorem}]
    We set $c=68c^*+12$, where $c^*$ is the constant in \cref{thm:ErdosPosa}.
    Let $r,h,\ell$ be positive integers.
    Let $H$ be an $h$-vertex graph that is a union of $r$ disjoint cycles of length at most $\ell$.
    We show that every graph $G$  with treewidth at least $ch \log (r+1)+cr\log r \log \ell$ contains $H$ as a minor. 
    We let $\Beta$ be a bramble of maximum order in $G$.
    By \cref{duality}, $\ord(\Beta) \geq 1+ch \log (r+1)+cr\log r \log \ell$.
    
    We proceed by induction on $r$.    
    If $r=1$, then \cref{cyclebramble} implies that there exists a cycle $C$ hitting all elements of $\Beta$, so $|V(C)| \geq \ord(\Beta) \geq h$, and hence $C$ contains $H$ as a minor.
    
    So, we can assume $r \ge 2$. 
    We enumerate the components of $H$ as $C_1,\ldots, C_r$ so that letting $\ell_i=|V(C_i)|$ we have $\ell \ge \ell_1 \ge \ell_2 \geq \ldots \ge \ell_r$.  

\medskip

\noindent \textbf{Case 1:} $G$ contains a cycle $C$ with length at least $\ell_1$ such that the order of $\Beta_{V(C)}$ is at most $6\ell_1$. 
\smallskip

\noindent\textsl{Proof of Case 1:}
 By \cref{small_hit} we have, $\tw(G-V(C)) \geq \tw(G) - 6\ell_1$.
    Since $c \log r \geq 6$,
    \[
    \tw(G) - 6\ell_1 \geq ch \log r+cr\log r \log \ell - 6\ell_1 \geq c(h-\ell_1) \log r+c(r-1) \log (r-1) \log \ell.
    \]
    So, by the induction hypothesis, $H-V(C_1)$ is a minor of $G-V(C)$.
    Since the length of $C$ is at least $\ell_1$, $C$ contains $C_1$ as a minor.
    So $H$ is a minor of $G$.\hfill$\blacksquare$

\medskip

\noindent \textbf{Case 2:} $G$ contains no cycle $C$ of length at least $\ell_1$ such that the order of $\Beta_{V(C)}$ is at most $6\ell_1$. 
\smallskip

We break \textbf{Case 2} into two subcases depending on the relation between $\ell_1$ and $r$.
\smallskip

\noindent \textbf{Case 2.1: $\ell_1 \ge 1+3c^* r\log r$.}   
\smallskip

\noindent\textsl{Proof of Case 2.1:}
Note that $ch\log r \geq 4\ell_1$.
Apply \cref{lem:pathpartition} with $t=2\ell_1$ to obtain the paths $P_1,P_2,Q_1, \dots , Q_{2\ell_1}$. 
If at least $\ell_1$ of those $Q_i$'s have $|V(Q_i)| \leq \ell_1$, then there are $1 \leq i_1<i_2 \leq 2\ell_1$ such that the subpath $P_1'$ of $P_1$ between the endpoints of $Q_{i_1}$ and $Q_{i_2}$ has length at least $\ell_1-1$.
So, the union $P_1' \cup P_2 \cup Q_{i_1} \cup Q_{i_2}$ contains a cycle $C$ containing $P_1'$ (which has length at least $\ell_1$) with $\ord(\Beta_{V(C)}) \leq |V(Q_1)|+|V(Q_2)|+ \ord(\Beta_{V(P_1)})+\ord(\Beta_{V(P_2)}) \leq 6\ell_1$, contradicting the assumption for \textbf{Case 2}.

So at least $\ell_1+1$ paths among $Q_1, \dots , Q_{2\ell_1}$ have at least $\ell_1+1$ vertices.
By symmetry, we may assume that $Q_1, \dots Q_{\ell_1+1}$ have at least $\ell_1+1$ vertices.
By ~\cref{lem:contracting_linkage_between_paths}, $P_1 \cup P_2 \cup \bigcup_{i=1}^{\ell_1+1} Q_i$ contains at least $r$ disjoint cycles of length at least $\min_{1 \leq i \leq \ell_1}|V(Q_i)| \geq \ell_1$.
So $G$ contains $H$ as a minor.
\hfill$\blacksquare$
\medskip

\noindent\textbf{Case 2.2: $\ell_1 < 1+3c^* r\log r$.}  
\smallskip

\noindent\textsl{Proof of Case 2.2:}
Let $a = \lfloor \frac{cr\log r}{4} \rfloor$.
By \cref{lem:pathpartition}, there exist two disjoint paths $P_1$ and $P_2$ in $G$ such that both $\Beta_{V(P_1)}$ and $\Beta_{V(P_2)}$ have order exactly $a$, and there exist disjoint paths $Q_1, \dots , Q_{a}$ each of which has an endpoint on each $P_i$ and is internally disjoint from $P_1 \cup P_2$. 

If at least $2+3c^* r\log r$ of those $Q_i$'s have $|V(Q_i)| \ge \ell_1$, then there are at least $r$ disjoint cycles of length at least $\ell_1$ contained in the union of $P_1 \cup P_2$ and those $Q_i$'s with $|V(Q_i)| \geq \ell_1$ by \cref{lem:contracting_linkage_between_paths}, which gives an $H$ minor in $G$.

Hence we may assume that there are at most $2+3c^* r\log r$ indices $i\in [a]$ such that $|V(Q_i)| \ge \ell_1$.

Let $b$ be the largest integer $i$ with $1 \leq i \leq \lceil 2^{3/4}\frac{3r}{4\ell_1} \rceil$ such that $\ell_i \geq 2^{-3/4}\ell_1$.
We claim that 
\begin{equation} \label{claim1}
	c \cdot \sum_{i=1}^b \ell_i \cdot \log r + cr\log r \log \ell - cr\log r \log \ell_{b+1} \geq \frac{3cr\log r}{4}.
\end{equation}
By the definition of $b$, we know $\sum_{i=1}^b \ell_i \geq b 2^{\frac{-3}{4}}\ell_1 $, so if $b = \lceil 2^{\frac{3}{4}} \frac{3r}{4\ell_1} \rceil$, then $\sum_{i=1}^b \ell_i \geq b 2^{\frac{-3}{4}} \ell_1 \geq \frac{3r}{4}$ and \cref{claim1} holds.
If $b \leq \lceil 2^{\frac{3}{4}} \frac{3r}{4\ell_1} \rceil -1$, then since $\ell_{b+1} < 2^{\frac{-3}{4}} \ell_1$, we have
\[
cr\log r \log \ell - cr\log r \log \ell_{b+1} \geq  cr\log r \log\frac{\ell_1}{\ell_{b+1}} \geq \frac{3cr\log r}{4}.
\]
This proves \cref{claim1}. 

Recall that there are at most $2+3c^* r\log r$ indices $i\in [a]$ such that $|V(Q_i)| \ge \ell_1$.
So at least $a-(2+3c^* r\log r)$ of those paths $Q_1,...,Q_a$ have less than $\ell_1$ vertices.
Now we claim that $a-(2+3c^* r\log r) > \ell_1 (2+3c^*b\log b)$.
If $\frac{r}{\ell_1} < 0.79$, then $b \leq \lceil 2^{3/4}\frac{3r}{4\ell_1} \rceil \leq 1$, so $\ell_1 (2+3c^*b\log b) \leq 2\ell_1$; if $\frac{r}{\ell_1} \geq 0.79$, then $b \leq \lceil 2^{3/4}\frac{3r}{4\ell_1} \rceil \leq \frac{2.6r}{\ell_1}$, so $\ell_1 (2+3c^*b\log b) \leq 2\ell_1 + 7.8c^* r\log r$.
Since $\ell_1 < 1+3c^* r\log r$, we have $\ell_1 (2+3c^*b\log b) \leq 2\ell_1 + 7.8c^* r\log r < 2 + 13.8c^* r\log r \leq a-(2+3c^* r\log r)$.

Hence, we can choose a set $S$ of $\lceil 2+3c^*b\log b \rceil$ paths such that every path in $S$ equals $Q_i$ for some $i \in [a]$ with $|V(Q_i)| < \ell_1$ and the distance between the endpoints in $P_1$ of any two paths in $S$ is at least $\ell_1$.
By \cref{lem:contracting_linkage_between_paths}, $P_1 \cup P_2 \cup \bigcup_{P \in S}P$ contains $b$ disjoint cycles $F_1,F_2,\dots,F_b$ each containing at least two paths in $S$.
Since the distance between any of the endpoints in $P_1$ of any two distinct paths in $S$ is at least $\ell_1$, each $F_i$ has a length of at least $\ell_1$.
So $\bigcup_{i=1}^bF_i$ contains $\bigcup_{i=1}^bC_i$ as a minor.

Moreover, since each path in $S$ has at most $\ell_1$ vertices,
	\begin{align*}
		\ord(\Beta_{\bigcup_{i=1}^bV(F_i)}) \leq {}& \ord(\Beta_{V(P_1 \cup P_2)}) + |S| \cdot \ell_1 \\
		\leq {}& 2a + (2+3c^*b\log b)\ell_1 + \ell_1 \\
		< {}& 2a + (a-(2+3c^* r\log r)) + (1+3c^* r\log r) \\
		= {}& 3a -1.
    \end{align*}
This implies 
    \begin{align*}
        \ord(\Beta_{\bigcup_{i=1}^bV(F_i)}) \leq 3a \leq \frac{3cr\log r}{4} \leq c \cdot \sum_{i=1}^b \ell_i \cdot \log r + cr\log r \log \ell - cr\log r \log \ell_{b+1}
    \end{align*}
by \cref{claim1}. 
Hence, by \cref{small_hit},
	\begin{align*}
		\tw(G-\bigcup_{i=1}^{b}V(F_i)) \geq {}& \tw(G) - (c \cdot \sum_{i=1}^b \ell_i \cdot \log r + cr\log r \log \ell - cr\log r \log \ell_{b+1}) \\
		\geq {}& c(h-|V(\bigcup_{i=1}^{b}C_i)|) \cdot \log (r+1)+cr\log r \log \ell_{b+1} \\
		= {}& c \cdot |V(\bigcup_{i=b+1}^{r}C_i)| \cdot \log (r+1)+cr\log r \log \ell_{b+1}.
	\end{align*}
Since the longest cycle in $\bigcup_{i=b+1}^{r}C_i$ has length at most $\ell_{b+1}$, the induction hypothesis implies that $G-\bigcup_{i=1}^{b}V(F_i)$ contains $\bigcup_{i=b+1}^rC_i$ as a minor.
Recall that $\bigcup_{i=1}^bF_i$ contains $\bigcup_{i=1}^bC_i$ as a minor.
Hence, $H$ is a minor of $G$.
\end{proof}

\bigskip

\noindent \textbf{Update:}
While a version of this paper was under review, Joret and Micek~\cite{JM} improved \cref{maintheorem} by showing $f(H)=O(|V(H)|+r\log r)$.

\bigskip

\noindent \textbf{Acknowledgement:}
This research was carried out at the first Pacific Rim Graph Theory Workshop held at the Institute of Mathematics of Academia Sinica in Taipei, Taiwan.
The authors thank the Institute for hosting and sponsoring the workshop.

\bibliographystyle{alphaurl}
\bibliography{treewidth}

\end{document}